\documentclass[draft,12pt]{article}

\usepackage[latin1]{inputenc}
\usepackage{amsmath,amssymb}
\usepackage{latexsym}

\usepackage{amsthm}

\newtheorem{theorem}{Theorem}[section]
\newtheorem{corollary}[theorem]{Corollary}
\newtheorem{lemma}[theorem]{Lemma}
\newtheorem{proposition}[theorem]{Proposition}
\newtheorem{definition}[theorem]{Definition}

\newtheorem{remark}[theorem]{Remark}

\numberwithin{equation}{section}

\def\dist{{\mathop {{\rm dist\, }}}}

\def\square{{\vcenter{\vbox{\hrule height.3pt
        \hbox{\vrule width.3pt height5pt \kern5pt
           \vrule width.3pt}
        \hrule height.3pt}}}}

\def\sD {{\cal D}} \def\sE {{\cal E}} \def\sF {{\cal F}}
  
  \def\sL {{\cal L}}

\def\wt{\widetilde}

\def\norm#1{{\Vert #1 \Vert}}
\def\del{{\partial}}
\def\lam{{\lambda}}

\def\bee{\begin{equation}}
\def\bet{\begin{theorem}}
\def\bep{\begin{proposition}}
\def\bel{\begin{lemma}}
\def\bec{\begin{corollary}}
\def\bed{\begin{definition}}
\def\ber{\begin{remark}}
\def\eee{\end{equation}}
\def\eet{\end{theorem}}
\def\eep{\end{proposition}}
\def\eel{\end{lemma}}
\def\eec{\end{corollary}}
\def\eed{\end{definition}}
\def\eer{\end{remark}}

\def\R{{\mathbb R}}

\def\Z{{\mathbb Z}}

\def\lam{{\lambda}}
\def\th{{\theta}}
\def\al{{\alpha}}
\def\grad{{\nabla}}

\def\eps{\varepsilon}
\def\vp{\varphi}

\def\norm#1{\Vert #1 \Vert}

 \def\qq {\qquad}
\def\del{{\partial}}
\def\wt{\widetilde}

\def\ni{\noindent }
\def\ms{\medskip}

\def\dist{{\mathop {{\rm dist\, }}}}

\def\square{{\vcenter{\vbox{\hrule height.3pt
        \hbox{\vrule width.3pt height5pt \kern5pt
           \vrule width.3pt}
        \hrule height.3pt}}}}

\def\tfrac#1#2{{\textstyle {\frac{#1}{#2}}}}

\def\tlint{{- \kern-0.85em \int \kern-0.2em}}  
\def\dlint{{- \kern-1.05em \int \kern-0.4em}}  

\def\sD {{\cal D}} \def\sE {{\cal E}} \def\sF {{\cal F}}
  
  \def\sL {{\cal L}}

\def\nn{{\nonumber}}

\begin{document}

\title{Meyers inequality and strong stability for stable-like operators}

\author{Richard F. Bass and Hua Ren\footnote{Research of both authors partially supported by NSF grant
DMS-0901505.}}

\date{\today}

\maketitle

\begin{abstract}
 Let $\al\in (0,2)$,
let
$$\sE(u,u)=\int_{\R^d}\int_{\R^d} (u(y)-u(x))^2\frac{A(x,y)}{|x-y|^{d+\al}}
\, dy\, dx$$
be the Dirichlet form for a stable-like operator, let
$$\Gamma u(x)=\Big(\int_{\R^d} (u(y)-u(x))^2\frac{A(x,y)}{|x-y|^{d+\al}}
\, dy\Big)^{1/2},$$
let $L$ be the associated infinitesimal generator, 
and  suppose $A(x,y)$ is jointly measurable, symmetric, bounded, and bounded below
by a positive constant. We prove that if $u$ is the weak solution to
$Lu=h$, then $\Gamma u\in L^p$ for some $p>2$.
This is the analogue of an inequality of Meyers for solutions to divergence
form elliptic equations.
 As an
application, we prove strong stability results for stable-like 
operators. If $A$ is perturbed slightly, we give explicit bounds
on how much the semigroup and fundamental solution are perturbed.

\end{abstract}

\section{Introduction}

Nowadays many researchers who use mathematical models consider situations
where discontinuities can occur. In analysis terms, this means they need to look
at integro-differential operators as well as differential operators. 
Integro-differential operators are not nearly as well understood
as their differential counterparts, and to study them it makes sense to first 
look at
the extreme case, that of purely integral operators.

In this paper we focus on a reasonably large class of such
integral operators, the stable-like operators. These are operators that
bear the same relationship to the fractional Laplacian as divergence
form operators do to the Laplacian.

To describe our results, let us first recall some facts about divergence
form operators. These have the form
$${\sL}_df(x)=\sum_{i,j=1}^d \frac{\del}{\del x_i}\Big(a_{ij}(\cdot) \frac{\del f}{\del x_j}
(\cdot)\Big)(x).$$
These have been studied even when the $a_{ij}$ are only bounded and
measurable, and to make sense of the operator in this case, one looks
at the corresponding Dirichlet form:
$$\sE_d(f,f)=\int_{\R^d} \sum_{i,j=1}^d a_{ij}(x)\frac{\del f}{\del x_i}(x)
\frac{\del f}{\del x_j}(x)\, dx.$$
One says that $u$ is a weak solution of ${\sL}_du=h$ if $\sE_d(u,v)=
-(h,v)$ for all $v$ in a suitably large class, where $(h,v)=\int_{\R^d}
h(x) v(x)\, dx$.

An inequality of Meyers (\cite{Meyers}) says that if the $a_{ij}$ are uniformly 
elliptic and $u$ is a weak
solution to ${\sL}_du=h$, then not only is $\grad u$ locally
in $L^2$ but it is locally in $L^p$ for some $p>2$. 

The Meyers inequality has many applications. One is to the stability
of solutions to ${\sL}_du=h$. Suppose one perturbs the coefficients $a_{ij}$
slightly. How does this affect the associated semigroup? What about the
fundamental solution associated with the operator ${\sL}_d$? These are 
natural questions since the coefficients $a_{ij}$ might themselves be only 
estimated or approximated. In \cite{Chenetal}
these issues were resolved, with an explicit bound on how large the
difference between the 
semigroups and solutions associated with two operators ${\sL}_d$ and $\wt {\sL}_d$
can be in terms of the difference of the coefficients $a_{ij}$ and
$\wt a_{ij}$.

Our purpose in this paper is to examine the analogues of these results
for stable-like processes. The operator we consider is
$${\sL} f(x)=\int_{\R^d} (f(y)-f(x))\frac{A(x,y)}{|x-y|^{d+\al}}\, dy,$$
where $\al\in (0,2)$ and $A(x,y)$ is bounded, symmetric, jointly measurable,
and bounded below. As in the case for divergence form operators, it is 
useful to
look at  the associated
Dirichlet form
$$\sE(f,f)=\int_{\R^d}\int_{\R^d} (f(y)-f(x))^2 \frac{A(x,y)}{|x-y|^{d+\al}}\, dy\, dx.$$

The bulk of this paper is devoted to proving a Meyers inequality for
weak solutions to ${\sL} u=h$ when $h$ is in $L^2$. Define
\bee\label{I1}
\Gamma u(x)=\Big( \int_{\R^d}\frac{(u(y)-u(x))^2}{|x-y|^{d+\alpha}}\, dy\Big)^{\frac12}.
\eee
Our main result is that there exists $p>2$ such that the $L^p$ norm of $\Gamma u$
is bounded in terms of the $L^2$ norms of $u$ and $h$; see Theorem \ref{meyers}.

Once one has the Meyers inequality for $\sE$, strong stability
results can be proved along the lines of \cite{Chenetal}. Suppose
$\wt \sE $ is defined in terms of $\wt A(x,y)$ analogously to \eqref{I1}. 
We obtain explicit bounds on the $L^p$ norm of $P_tf-\wt P_tf$
and on the $L^\infty$ norm of $p(t,x,y)-\wt p(t,x,y)$ in terms
of
 $$G(x)=\sup_{y\in R^d}|A(x,y)-\wt{A}(x,y)|,$$ where $P_t$ and $p(t, \cdot,
\cdot)$ are the semigroup and fundamental solution associated with ${\sL}$
and $\wt P_t$ and $\wt p(t,\cdot,\cdot)$ are defined similarly. See Theorems
\ref{THE3.2}, \ref{THE4.2A}, and \ref{THE4.3A}.

Our proof of the Meyers inequality begins by first
proving
a Caccioppoli inequality.
However there are considerable differences
between the stable-like case and the divergence form case. 
For example, as one would expect,
our Caccioppoli inequality is not a local one; the integral of $|\Gamma u|^2$
on a ball depends on values of $u$ far outside the ball. This makes proving
the Meyers inequality considerably more difficult and requires
the introduction of some new ideas, such as localization, use
of the Hardy-Littlewood maximal function, and use of the Sobolev-Besov
embedding theorem.

For other papers on stable-like operators and on closely related operators,
see \cite{MGK} --
\cite{bkk}, 
\cite{Chen11} --
\cite{CKarxiv}, \cite{foondun1}, \cite{foondun2},
\cite{kolo}, and \cite{song-vondracek}.

\ni {\bf Acknowledgment.} We would like to thank M.~Kassmann for some very 
helpful discussions.

There is an error in the published version of this paper. The scaling argument
appealed to in the line following (4.11) does not work. This was pointed out
to us by  T.~Mengesha. What is needed is to restrict attention to $R$ less than $4\sqrt d$.
This version of the paper corrects the error.

\section{Preliminaries}\label{prelim}

We use the letter $c$ with or without subscripts to denote a finite
positive constant whose exact value is unimportant and which can vary from
place to place.
We use $B(x,r)$ for the open ball in $\R^d$ with center $x$ and radius $r$.
When the center is clear from the context, we will also write $B_r$. The Lebesgue
measure of $B(x,r)$ will be denoted $|B(x,r)|$.
We write $(u,v)$ for $\int_{\R^d} u(x)v(x)\, dx$.

Let $\al\in (0,2)$ and suppose the dimension $d$ is greater than $\al$.
We let $A(x,y)$ be a jointly measurable symmetric function on $\R^d\times \R^d$
and suppose
there exists $\Lambda>0$ such that
$$\Lambda^{-1}\le A(x,y)\le \Lambda, \qq x,y\in \R^d.$$

We define the Dirichlet form $\sE$ with domain $\sD(\sE)=\sF$  by
\begin{align}
\sE(u,v)&=\int_{\R^d}\int_{\R^d}(u(y)-u(x))(v(y)-v(x))\frac{A(x,y)}{|x-y|^{d+\alpha}}\, dy\, dx,\phantom{\Bigg[}\label{dfno}\\
\sF&=\{u\in L^2(\R^d):\sE(u,u)<\infty\}.\nn  
\end{align}

Observe that $\sF=W^{\alpha/2,2}(\R^d)$, the fractional
Sobolev space of order $\al/2$,
defined by
\begin{align*}
W^{{\alpha/2},2}(\R^d)
=\Big\{u\in L^2(\R^d): \int_{\R^d}\int_{\R^d}\frac{(u(x)-u(y))^2}{|x-y|^{d+\alpha}}\, dy\, dx<{\infty}\Big\}.
\end{align*}
See \cite{Adams} for more details. It is well known that  $(\sE,\sF)$ is a  
regular Dirichlet form on $L^2(\R^d)$. The strong Markov symmetric process $X$ 
 associated with $(\sE,\sF)$ is called a  stable-like process.
Let $\{P_t\}_{t\ge 0}$  be the semigroup corresponding to $(\sE,\sF)$. 

For $u\in \sF$ define
\bee\label{Pr-E101}
\Gamma u(x)= \Big(\int_{\R^d}\frac{(u(y)-u(x))^2}{|x-y|^{d+\alpha}}\, dy\Big)^{\frac12}.
\eee
Since $\int |\Gamma u(x)|^2\, dx=\sE(u,u)<\infty$, then $\Gamma u\in L^2$,
and in particular $\Gamma u(x)$ exists for almost every $x$.

Let ${\sL}$ be the infinitesimal generator corresponding to $\sE$ (see 
\cite{FOT}). There are a number of known results that follow from the
spectral theorem.  We collect these in the following lemma  for the convenience
of the reader.

\begin{lemma}\label{Grandlemma} (1) For $t>0,\ f \in L^2(D)$, we have
\[\sE(P_tf,P_tf)\le ct^{-1}\|f\|_2^2.\]
(2) If $g\in L^2$, then $P_t g$ is in $\sD({\sL})$,
the domain of $\sL$.\\
(3) If $f,g\in \sF$, then 
$$\frac{d}{dt}(P_tf,g)=-\sE(P_tf,g).$$
(4)  If $f\in \sF$, then 
\bee\label{Pr44}
\sE(P_tf,P_tf)\le \sE(f,f).
\eee
\end{lemma}

The proof of this lemma is given in Section \ref{proofoflemma}.

\section{Caccioppoli inequality}

In this section, we will derive a Caccioppoli inequality for the weak solution of the equation
\begin{align}
&\sL u(x)= h(x),\qq x\in \R^d,
\label{BE01}
\end{align}
where $h\in L^2(\R^d)$.
A function $u\in W^{\frac{\alpha}{2},2}(\R^d)$ is called a weak solution of  \eqref{BE01} if
\begin{align}
\sE(u,v)= -(h,v)\ \ \textnormal{for all}\ v\in W^{\frac{\alpha}{2},2}(\R^d),
\label{BE02}
\end{align}
where $(h,v)=\int h(x)v(x)\, dx$.

\begin{theorem}\label{THM2.1} Let $x_0\in \R^d$. Suppose $u(x)$ satisfies
 \eqref{BE02}. There exists a constant $c_1$ depending only
on $\Lambda, \al,$ and $d$  such that
\begin{align}
\int_{B_{R/2}}\int_{\R^d}&(u(y)-u(x))^2\frac{A(x,y)}{|x-y|^{d+\alpha}}\, dy\, dx\nn\\
&\le c_1\int_{\R^d}u^2(y)\psi(y)\, dy+ \int_{B_R}|h(y)u(y)|\, dy,
\label{DE1}
\end{align}
where 
\[\psi(x)= {R^{-\alpha}}\wedge\frac{R^d}{|x-x_0|^{d+\alpha}}.\]
\end{theorem}

\begin{proof} We define a cutoff function $\varphi(x): \R^d\rightarrow [0,1]$ such that $\varphi= 1$ on $B_{R/2}$, $\varphi= 0$ on $ B_R^c$, and
\[|\varphi(x)-\varphi(y)|\le c\frac{|x-y|}{R}.\]
For example, we can take
\[\varphi(x)= 1-\Big(\frac{\dist(x,B(x_0,R/2))}{R/2}\wedge 1\Big).\]
In what follows the constants may depend on $R$.

Let $v(x)= \varphi^2(x)u(x)$.  Since $|v|\le |u|$ and $u\in L^2$, then $v\in L^2$. Since
\[v(y)-v(x)= (u(y)-u(x))\varphi^2(y)+ u(x)(\varphi^2(y)-\varphi^2(x)),\]
then
\begin{align*}
\int_{\R^d}\int_{\R^d}\frac{(v(y)-v(x))^2}{|x-y|^{d+\alpha}}\, dy\, dx&\le 2\int_{\R^d}\int_{\R^d}\frac{(u(y)-u(x))^2\varphi^4(y)}{|x-y|^{d+\alpha}}\, dy\, dx\\
&\qq +2\int_{\R^d}\int_{\R^d}\frac{u^2(x)(\varphi^2(y)-\varphi^2(x))^2}{|x-y|^{d+\alpha}}\, dy\, dx.\\
\end{align*}
The first term on the right hand side is finite because $\varphi\le 1$ and $u\in \sF$. The second term is bounded by
\[c\int_{\R^d}\int_{\R^d}\frac{u^2(x)(1\wedge |y-x|^2/R^2)}{|x-y|^{d+\alpha}}\, dy\, dx\le c\int_{\R^d}u^2(x)\, dx,\]
which is finite since $u\in L^2$.
Therefore $v\in \sF$.

We write
\begin{align*}
-(h,v)&=\sE(u,v)\\
&= \int_{\R^d}\int_{\R^d}(u(y)-u(x))(\varphi^2(y)u(y)-\varphi^2(x)u(x))
\frac{A(x,y)}{|x-y|^{d+\alpha}}\, dy\, dx\\
&= \int_{\R^d}\int_{\R^d}(u(y)-u(x))^2\varphi^2(x)\frac{A(x,y)}{|x-y|^{d+\alpha}}\, dy\, dx\\
&\qq+ \int_{\R^d}\int_{\R^d}[(u(y)-u(x))(\varphi(y)-\varphi(x))(\varphi(y)+\varphi(x))u(y)]\\
&\qq \qq \times \frac{A(x,y)}{|x-y|^{d+\alpha}}\, dy\, dx\\
&= I_1 - I_2.\\
\end{align*}
Then 
\begin{align}
I_1
&= I_2 -\int_{\R^d}h(y)\varphi^2(y)u(y)\, dy\nn\\
&\le I_2+ \int_{B_R}|h(y)u(y)|\, dy.
\label{DE0}
\end{align}

Using the inequality $ab\le \frac18a^2+ 2b^2$, symmetry, and the fact that $0\le \varphi(x)\le 1$, we have
\begin{align*}
I_2
&\le \tfrac18\int_{\R^d}\int_{\R^d}(u(y)-u(x))^2(\varphi(y)+\varphi(x))^2
\frac{A(x,y)}{|x-y|^{d+\alpha}}\, dy\, dx\\
&\qq + 2\int_{\R^d}\int_{\R^d}(\varphi(y)-\varphi(x))^2u^2(y)\frac{A(x,y)}{|x-y|^{d+\alpha}}\, dy\, dx\\
&\le \tfrac12\int_{\R^d}\int_{\R^d}(u(y)-u(x))^2\varphi^2(x)\frac{A(x,y)}{|x-y|^{d+\alpha}}\, dy\, dx\\
&\qq + 2\int_{\R^d}\int_{\R^d}(\varphi(y)-\varphi(x))^2u^2(y)\frac{A(x,y)}{|x-y|^{d+\alpha}}\, dy\, dx\\
&= \tfrac12 I_1+ 2\int_{\R^d}\int_{\R^d}(\varphi(y)-\varphi(x))^2u^2(y)\frac{A(x,y)}{|x-y|^{d+\alpha}}\, dy\, dx.\\
\end{align*}
Therefore                
\begin{align}
\tfrac12 I_1\le 2\int_{\R^d}\int_{\R^d}(\varphi(y)-\varphi(x))^2&u^2(y)\frac{A(x,y)}{|x-y|^{d+\alpha}}\, dy\, dx\nn\\
&+ \int_{B_R}|h(y)u(y)|\, dy.
\label{DE3}
\end{align}

Next, using $|\vp(y)-\vp(x)|\le c(1\land |x-y|/R)$, some calculus shows that
\bee\label{E2.5A}
\int_{\R^d}(\varphi(y)-\varphi(x))^2\frac{A(x,y)}{|x-y|^{d+\alpha}}\, dx\le cR^{-\alpha}, \qq y\in \R^d.
\eee
If $y\notin B_{2R}$, then
$$\int_{\R^d} (\vp(y)-\vp(x))^2\frac{A(x,y)}{|x-y|^{d+\al}}\, dx
\le c\int_{B_R} \frac{dx}{|y-x_0|^{d+\al}}=c\frac{R^d}{|y-x_0|^{d+\al}}.$$
Hence the first term on the right hand side of \eqref{DE3} is bounded by
\bee\label{E2.5B}
c\int u(y)^2 \psi(y)\, dy.
\eee

Combining \eqref{DE3} and \eqref{E2.5B} with the fact that
$$I_1 \ge \int_{B_{R/2}}\int_{\R^d}(u(y)-u(x))^2\frac{A(x,y)}{|x-y|^{d+\alpha}}\, dy\, dx$$
completes the proof.
\end{proof}

For another approach to the Caccioppoli inequality for non-local operators,
see \cite{Kassmann}.

\section{Meyers inequality}

Let $h\in L^2$. We consider the weak solution $u(x)$ of \eqref{BE02}:
 \[\sE(u, v)= -(h, v)\]
for all $ v\in W^{\frac{\alpha}{2},2}(\R^d).$
We will show that $\Gamma u$ is in $L^p$ for some $p> 2$. 
We suppose throughout this section that $d>\al$. This will always be
the case if $d\ge 2$.

Let $$ u_R= \frac{1}{|B_R|}\int_{B_R}u(y)\, dy.$$ 
Using Theorem \ref{THM2.1} with 
$u$ replaced by $u- u_R$, we have
\begin{align}
\|\Gamma u\|_{L^2(B_{R/2})}^2&\le c\int_{\R^d}(u(x)-u_R)^2\psi(x)\, dx\label{CE4}\\
&\qq+ \int_{B_R}|h(x)(u(x)-u_R)|\, dx.\nn
\end{align}

\bel\label{PRO1} Suppose $u\in W^{\frac{\alpha}{2},q}(B_R),$ $1<q\le 2$.
Suppose $x_0\in \R^d$ and $R>0$.
Let $p=2dq/(2d-q\al)$.  Then $u\in L^p(B_R)$  and there exists a constant $c_1$
depending only on $d, \al,$ and $q$  such that
\bee\label{E101}
\|u-u_R\|_{L^p(B_R)}\le c_1\Big[\int_{B_R}\int_{B_R}\frac{(u(y)-u(x))^q}{|x-y|^{d+\frac{\alpha}{2}q}}\, dy\
 dx\Big]^{\frac1q}.
\eee
\eel

\begin{proof} We first do the case $R= 1$. 
By the Sobolev-Besov embedding theorem (see Theorem 7.57 in \cite{Adams} or  Section 2.3.3 in \cite{Edmunds-Triebel}), we know 
\begin{align}
\|u-u_R\|_{L^p(B_1)}&\leq c\|u-u_R\|_{W^{\frac{\alpha}{2},q}(B_1)}
\label{BE12}\\
&=c\Big\{\|u-u_R\|_{L^q(B_1)}+
\Big[\int_{B_1}\int_{B_1}\frac{(u(y)-u(x))^q}{|x-y|^{d+\frac{\alpha}{2}q}}\, dy\, dx\Big]^{\frac{1}{q}}.\Big\}\nn
\end{align}
On the other hand, the fractional Poincar\'e inequality for $u\in W^{\frac{\alpha}{2},q}(B_1)$ (see equation (4.2) in \cite{Mingione})
tells us
\bee\label{BE13}
\|u-u_R\|_{L^q(B_1)}\leq c \Big[\int_{B_1}\int_{B_1}\frac{(u(y)-u(x))^q}{|x-y|^{d+\frac{\alpha}{2}q}}\, dy\, dx\Big]^{\frac{1}{q}}.
\eee 
Combining \eqref{BE12} and \eqref{BE13} proves the lemma in the case $R=1$.

The case for general $R$ follows by a  scaling argument, that is, by a change
of variables. 
The $dy\, dx$ expression in the right hand side of \eqref{E101} contributes a factor $R^{2d}$ and 
the denominator  contributes
a factor 
$R^{-(d+\al q/2)}$, so the right hand side of \eqref{E101} is equal to
$$c(R^{d-\al q/2})^{1/q} 
\Big[\int_{B_1}\int_{B_1}\frac{(v(y)-v(x))^q}{|x-y|^{d+\frac{\alpha}{2}q}}\, dy\
 dx\Big]^{\frac1q},$$
where $v(z)=u(Rz)$.
Similarly the left hand side of \eqref{E101}
is equal to  $$R^{d/p}
\|v-v_1\|_{L^p(B_1)}.$$
Inequality \eqref{E101} then follows by the preceding paragraph and our choice of $p$.
\end{proof}

\begin{proposition}\label{PRO2.2} There exists $q_1 \in (1,2)$ and a
constant $c_1$ depending on $d, \al,$ and $q_1$  such that
if $x_0\in \R^d$ and $R>0$, then
\begin{align}
\|u-u_R\|_{L^2(B_R)}\le c R^{(\alpha-\alpha_1)/2}\|\Gamma u\|_{L^{q_1}(B_R)},
\label{BE15}
\end{align}
where $\al_1=(2-q_1)d/q_1$.
\end{proposition}

\begin{proof}
Again we may suppose $R=1$ and obtain the general case by a scaling
argument as in the last paragraph of the proof of Lemma \ref{PRO1}.
Take  $\alpha_1<\alpha$ and let $q_1=2d/(d+\al_1)$. Note that $q_1\in (1,2)$.
By Lemma \ref{PRO1}
\begin{align}
\|u-u_R\|_{L^2(B_R)}\le c\Big[\int_{B_R}\int_{B_R}\frac{(u(y)-u(x))^ {q_1}}{|x-y|^{d+{{\alpha_1 q_1/2}}}}\, dy\, dx\Big]^{\frac{1}{q_1}}.
\label{BE16}
\end{align}
Fix $x$ for the moment.
Using H\"older's inequality with respect to the measure $|x-y|^{-d} \, dy$,
\begin{align*}
\int_{B_R}&\frac{(u(y)-u(x))^ {q_1}}{|x-y|^{d+{{\alpha_1q_1}/{2}}}}\, dy\\
&=\int_{B_R}\frac{(u(y)-u(x))^{q_1}}{|x-y|^{\alpha q_1/2}}
\frac{1}{|x-y|^{{(\alpha_1-\alpha)}{q_1}/{2}}}\frac{1}{|x-y|^d}\,dy\\
&\le \Big[\int_{B_R}\Big(\frac{(u(y)-u(x))^{q_1}}{|x-y|^{\alpha{q_1}/{2}}}\Big)
^{\frac{2}{q_1}}\frac{1}{|x-y|^d}\, dy\Big]^{\frac{q_1}{2}}\\
&\qq\qq\times\Big[\int_{B_R}\Big(\frac{1}{|x-y|^{(\alpha_1-\alpha){q_1}/{2}}}\Big)
^{\frac{2}{2-q_1}}\frac{1}{|x-y|^d}\, dy\Big]^{\frac{2-q_1}{2}}\\
&= \Big[\int_{B_R}\frac{(u(y)-u(x))^2}{|x-y|^{d+\alpha}}\, dy\Big]^{\frac{q_1}{2}}
\Big[\int_{B_R}\frac{1}{|x-y|^{(\alpha_1-\alpha)\frac{q_1}{2-q_1}+d}}\, dy\Big]^{\frac{2-q_1}{2}}\\
&\le c \Big[\int_{B_R}\frac{(u(y)-u(x))^2}{|x-y|^{d+\alpha}}\, dy\Big]^{\frac{q_1}{2}}\\
&\le c  |\Gamma u(x)|^{q_1}.
\end{align*}
Integrating over $x\in B_R$, taking the $q_1^{th}$ root, and combining
with  \eqref{BE16}
yields \eqref{BE15}.
\end{proof}

\begin{proposition}\label{THE2.2} There exists $p\in (2, 4d/(2d-\al))$ and
a constant $c_1$ depending on $\Lambda$, $d$, $\al$, and $p$
 such that if $u$ satisfies \eqref{BE02}, then 
$$\|\Gamma u\|_{L^p(\R^d)}\le c_1\Big(\sE(u, u)^{\frac12}+ \|h\|_{L^2(\R^d)}
+\norm{u}_{L^p(\R^d)} +\norm{u}_{L^{2p/(4-p)}(\R^d)} \Big).$$ 
\end{proposition}

\begin{proof} 

\ni{\sl Proof.} Set $x_0=0$ for now.
From \eqref{CE4} we know that
\begin{align}
\|\Gamma u\|_{L^2(B_{R/2})}^2&\le c\int_{\R^d}(u(x)-u_R)^2\psi_R(x)\, dx+ \int_{B_R}|h(x)(u(x)-u_R)|\, dx\nn\\
&\le {c}R^{-\al}\int_{B_R}(u(x)-u_R)^2\, dx+ c\int_{B_R^c}u(x)^2\psi_R(x)\, dx\nn\\
&\qq + c\int_{B_R^c} u_R^2\psi_R(x)\, dx+ \int_{B_R}|h(x)(u(x)-u_R)|\, dx\nn\\
&=J_1+J_2+J_3+J_4;
\label{A.1}
\end{align}
recall 
$$\psi_R(x)=R^{-\al}\land \frac{R^d}{|x-x_0|^{d+\al}}.$$
We proceed to bound $J_1, J_2, J_3$, and $J_4$.

Using Proposition \ref{PRO2.2}, we have
\bee
\label{A.2}
J_1\le cR^{-\al_1} \Big(\int_{B_R}\Gamma u(x)^{q_1}\, dx\Big)^{\frac{2}{q_1}}
\eee
for $q_1\in (1,2)$.

Let $M$ be the Hardy-Littlewood maximal operator:
$$Mf(x)=\sup_{r>0} \frac{1}{|B_r|}\int _{B(x,r)} |f(z)|\, dz.$$
If $y\in B_{R}$ and $k\ge 0$, then 
$$\frac{1}{|B_{2^kR}|}\int_{B_{2^kR}}u(x)^2\, dx
\le \frac{1}{|B_{2^kR}|}\int_{B(y,(2^k+1)R)} u(x)^2\, dx
\le cM(u^2)(y).$$ 
Similarly, if $x\in B_R$,
$$|u_R|\le \frac1{|B_R|}\int_{B_R} |u(z)|\, dz\le cMu(x).$$


For $J_2$, we then have
\begin{align*}
c\int_{B_R^c}u(x)^2\psi_R(x)\, dx&= c\sum_{n=0}^\infty \int_{\{B_{2^{n+1}R}
-B_{2^nR}\}} u(x)^2 \frac{R^d}{|x|^{d+\al}} dx\\
&\le c\sum_{n=0}^\infty \frac{(2^{n+1}R)^d}{|B_{2^{n+1}R}|}
\int_{B_{2^{n+1}R}} u(x)^2\, dx\, \frac{R^d}{(2^nR)^{d+\al}}\\
&\le c\sum_{n=0}^\infty M(u^2)(y) 2^{nd}\frac{R^{2d}}{2^{nd}2^{n\al}
R^{d+\al}}\\
&=cM(u^2)(y) R^{d-\al}\sum_{n=0}^\infty \frac{1}{2^{n\al}}\\
&=c M(u^2)(y)R^{d-\al},
\end{align*}
as long as $y\in B_R$.

For $J_3$ we have
\begin{align*}
c\int_{B_R^c} u_R^2\psi_R(x)\, dx&=cu_R^2\int_{B_R^c}  \frac{R^d}{|x|^{d+\al}}
\, dx\\
&=cR^{d-\al} u_R^2\le cR^{d-\al}M(u^2)(y)
\end{align*}
if $y\in B_R$.

Since
$|B(x,s)|^{-1}\int_{B(x,s)} u(y)\, dy$ converges to $u(x)$ as $s\to 0$
for almost
every $x$ and is bounded by $Mu(x)$, we have $|u(x)|\le Mu(x)$ a.e. 
Thus, with  $x\in B_R$,
\begin{align*}
J_4=c\int_{B_R} |h(x)(u(x)-u_R)|\, dx
&\le c\int _{B_R} |h(x) u(x)|\, dx\\
&\qq +c\int_{B_R} |h(x) Mu(x)|\, dx\\
&\le c\int _{B_R} |h(x)| Mu(x)\, dx.
\end{align*}

Combining our bounds for $J_1, J_2, J_3$, and $J_4$,
if $y\in B_R$,
\begin{align}
\norm{\Gamma u}^2_{L^2(B_{R/2})}
&\le cR^{-\al_1} \norm{\Gamma u}^2_{L^{q_1}(B_R)}
+c R^{d-\al}M(u^2)(y)\label{A.3}\\
&\qq + c\int_{B_R} |h(x)| Mu(x)\, dx.\nn
\end{align}

Integrating both sides of \eqref{A.3}  over $y\in B_R$ and
dividing by $|B_R|$, we conclude that
\begin{align}
\int_{B_{R/2}}\Gamma u(x)^2\, dx &\le {c}R^{-\al_1}\Big(\int_{B_R}\Gamma u(x)^{q_1}\, dx\Big)^{\frac{2}{q_1}}\label{A.4}\\
&\qq + {c}R^{d-\al}\int_{B_R}M(u^2)(x)\, dx+ c\int_{B_R}|h(x)|Mu(x)\, dx.\nn
\end{align}

Let $$g(x)= \Gamma u(x)^{q_1}$$ and $$f(x)= \Big(M(u^2)(x)+ |h(x)|Mu(x)\Big)^{\frac{q_1}{2}}.$$
Set $R_0= 4\sqrt d$ and suppose from now on that $R< R_0$. 
Recall that we assume $d\ge \al$ (see the second paragraph of Section \ref{prelim}). 
Noting that $R^{d-\al}$ is then bounded
by $(4\sqrt d)^{d-\al}$ and recalling that $\al_1=(2-q_1)d/q_1$,
we can rewrite \eqref{A.4} as 
\begin{align}
&\frac{1}{|B(x_0,R)|}\int_{B(x_0,R/2)}g^{\frac{2}{q_1}}(x)\, dx\label{A.5}\\
&\qq\le c\Big(\frac{1}{|B(x_0,R)|}\int_{B(x_0,R)}g(x)\, dx\Big)^{\frac{2}{q_1}}+ c\frac{1}{|B(x_0,R)|}\int_{B(x_0,R)}f^{\frac{2}{q_1}}(x)\, dx\nn
\end{align}
if $R<R_0$.
By a translation argument, \eqref{A.5} holds for 
all $x_0\in \R^d$.

We now apply the reverse H\"older inequality (see Theorem 4.1 in [12]). 
Thus  there exists $\varepsilon> 0$ and $c_1> 0$ such that 
if $R<R_0$, then $g(x)\in
L^t(B(x_0,R/2))$ for all $t\in [\frac{2}{q_1},\frac{2}{q_1}+\varepsilon)$
and   
\begin{align*}
\Big(\frac{1}{|B(x_0,R/2)|}\int_{B(x_0,R/2)}g^{t}(x)\, dx\Big)^{\frac{1}{t}}&\le c
\Big(\frac{1}{|B(x_0,R)|}\int_{B(x_0,R)}g^{\frac{2}{q_1}}(x)\, dx\Big)^{\frac{q_1}{2}}\\
&\qq + c\Big(\frac{1}{|B(x_0,R)|}\int_{B(x_0,R)}f^{t}(x)\, dx\Big)^{\frac{1}{t}}.
\end{align*}
This leads to
\begin{align*}
\Big(\frac{1}{|B(x_0,R/2)|}&\int_{B(x_0,R/2)}\Gamma u(x)^{q_1{t}}\, dx\Big)^{\frac{1}{t}}\\
&\le c\Big(\frac{1}{|B(x_0,R)|}\int_{B(x_0,R)}\Gamma u(x)^2\, dx\Big)^{\frac{q_1}{2}}\\
&\qq+
c\Big(\frac{1}{|B(x_0,R)|}\int_{B(x_0,R)}(M(u^2))^{tq_1/2}(x)\, dx\Big)^{\frac{1}{t}}\\
&\qq +c\Big(\frac{1}{|B(x_0,R)|} \int (|h|Mu)^{tq_1/2}\, dx\Big)^{1/t}.
\end{align*}
Choose $t\in (2/q_1, 2/q_1+\eps)$ so that $q_1t<4d/(d-\al)$ and set $p=q_1t$.

Now set $R=2\sqrt d$ for the remainder of the proof.
Taking $q_1^{th}$ roots and using 
the inequality $(a+b)^{1/q_1}\le a^{1/q_1}+b^{1/q_1}$,
\begin{align*}
\norm{\Gamma u}_{L^p(B(x_0,R/2))}&\le c\norm{\Gamma u}_{L^2(B(x_0,R))}
+c\norm{M(u^2)}^{1/2}_{L^{p/2}(B(x_0,R))}\\
&\qq + c\norm{h(Mu)}^{1/2}_{L^{p/2}(B(x_0,R))}.
\end{align*}


For $k\in \Z^d$, let $C_k=B(k, \sqrt d)$ and $D_k=B(k, 2\sqrt d)$.
Note that $\R^d\subset \cup_{k\in \Z^d} C_k$ and that there exists
an integer $N$ depending only on the dimension $d$ such that no point
of $\R^d$ is in more than $N$ of the $D_k$. This can be
expressed as $\sum_{k\in \Z^d} \chi_{D_k}\le N$.

Using $\sum a_k^{p/2}\le (\sum a_k)^{p/2}$
when each  $a_k\ge 0$ and $p/2\ge 1$, we write
\begin{align*}
\int _{\R^d} |\Gamma u(x)|^p\, dx&\le \sum_{k\in \Z^d} \int_{C_k}
|\Gamma u(x)|^p\, dx\\
&\le c\sum_k \Big(\int_{D_k} |\Gamma u(x)|^2\, dx\Big)^{p/2}
+c\sum_k \int_{D_k} (M(u^2)(x))^{p/2}\, dx\\
&\qq + c\sum_k \int_{D_k} (|h(x)|Mu(x))^{p/2}\, dx\\
&\le c\Big(\sum_k \int_{D_k} |\Gamma u(x)|^2\, dx\Big)^{p/2}
+c\sum _k \int_{D_k} (M(u^2)(x))^{p/2}\, dx\\
&\qq + c\sum_k \int_{D_k} (|h(x)|Mu(x))^{p/2}\, dx\\
&=c\Big(\int_{\R^d} |\Gamma u(x)|^2 \sum_k \chi_{D_k}(x)\, dx
\Big)^{p/2}\\
&\qq +c\int_{\R^d} (M(u^2)(x))^{p/2}\sum_k \chi_{D_k}(x)\, dx\\
&\qq +c\int_{\R^d} (|h(x)|Mu(x))^{p/2}\sum_k \chi_{D_k}(x)\, dx.
\end{align*}
We thus obtain
\begin{align}
\int_{\R^d} |\Gamma u|^p&\le c\Big(\int_{\R^d} |\Gamma u|^2\,dx\Big)^{p/2}
+ c\int_{\R^d} (M(u^2))^{p/2}\,dx\nn\\
&\qq +c\int_{\R^d} (|h|\, Mu)^{p/2}\,dx. \label{E85}
\end{align}

Letting $r=4/p$ and $s=4/(4-p)$, H\"older's inequality and the inequality
$ab\le \frac12 a^2+\frac12 b^2$ shows
\begin{align}
\int(|h| Mu)^{p/2}&\le \Big(\int |h|^{pr/2}\Big)^{1/r}
\Big(\int (Mu)^{ps/2}\Big)^{1/s}\nn\\
&\le \tfrac12 \Big(\int |h|^2\Big)^{p/2} 
 +\tfrac12 \Big(\int (Mu)^{2p/(4-p)}
\Big)^{(4-p)/2}.
\label{E86}
\end{align}
Since $M$ is a bounded operator on $L^{p'}$ for each $p'>1$ and 
we know that $2p/(4-p)>1$,
the second term on the last line of \eqref{E86} is bounded
by
$$c\Big(\int |u|^{2p/(4-p)}\Big)^{(4-p)/2}.$$
Similarly, since $p>2$, the second term on the right hand side of 
the first line of \eqref{E85}
is bounded by
$$c\int (|u|^2)^{p/2}=c\int |u|^p.$$
Therefore
\begin{align*}
\int_{\R^d} |\Gamma u|^p&\le c\Big(\int |\Gamma u|^2\Big)^{p/2}
+c\int |u|^p+c\Big(\int |h|^2\Big)^{p/2}\\
&\qq+ c\Big(\int |u|^{2p/(4-p)}\Big)^{(4-p)/2}.
\end{align*}
Taking $p^{th}$ roots and using $(a+b)^{1/p}\le a^{1/p}+b^{1/p}$, we obtain
\begin{align*}
\norm{\Gamma u}_{L^p(\R^d)}&
\le c\norm{\Gamma u}_{L^2(\R^d)}+c \norm{u}_{L^p(\R^d)}
+c\norm{h}_{L^2(\R^d)}\\
&\qq +c\norm{u}_{L^{2p/(4-p)}(\R^d)}.
\end{align*}
This completes the proof of the proposition.
\end{proof}

We now bound the $L^p$ and $L^{2p/(4-p)}$ norms of $u$.

\bet\label{meyers} (1) Suppose $d>\al$ and \eqref{BE02} holds. There
exists $p>2$ and a constant $c_1$ depending on $\Lambda, p, d$, and $\al$
such that
$$\|\Gamma u\|_{L^p(\R^d)}\le c_1\Big(\sE(u, u)^{\frac12}+ \|h\|_{L^2(\R^d)}
+\norm{u}_{L^2(\R^d)}\Big).$$
(2) If in addition $u\in \sD({\sL})$, there exists a constant $c_2$
such that
$$\|\Gamma u\|_{L^p(\R^d)}\le c_1\Big( \|h\|_{L^2(\R^d)}
+\norm{u}_{L^2(\R^d)} \Big).$$
\eet

\begin{proof}
Let $p_1=2d/(d-\al)$. Let $C_k$  be defined as in the previous
proof. 

By Lemma \ref{PRO1} with $q=2$
$$\int_{C_k} |u-u_{C_k}|^{p_1}\le c\Big(\int_{C_k} |\Gamma u(x)|^2\, dx\Big)^{{p_1}/2}.$$
Here $u_{C_k}=(1/|C_k|)\int_{C_k} u$. Then
\begin{align*}
\sum_{k\in \Z^d} \int_{C_k} |u-u_{C_k}|^{p_1}&
\le c\sum_k \Big(\int_{C_k}|\Gamma u(x)|^2\, dx\Big)^{{p_1}/2}\\
&\le c\Big(\sum_k \int_{C_k} |\Gamma u(x)|^2\, dx\Big)^{{p_1}/2}\\
&\le c\Big(\int_{\R^d} |\Gamma u(x)|^2\,\sum_k \chi_{C_k}(x)\, dx\Big)^{{p_1}/2}\\
&\le c\Big(\int_{\R^d} |\Gamma u(x)|^2\, dx\Big)^{{p_1}/2}.
\end{align*}
Also,
$$\int_{C_k} |u_{C_k}|^{p_1}=c|u_{C_k}|^{p_1}\le c\Big(\int_{C_k} |u|^2\Big)^{p_1/2}$$
by Jensen's inequality. Similarly to the above, 
$$\sum_{k} \int_{C_k} |u_{C_k}|^{p_1}\le c\Big(\int_{\R^d} u^2\Big)^{p_1/2}.$$
Hence
\begin{align*}
\int |u|^{p_1}\le \sum_k \int_{C_k} |u|^{p_1}
&\le c\sum_k \int_{C_k} |u-u_{C_k}|^{p_1}+\sum_k \int_{C_k} |u_{C_k}|^{p_1}\\
&\le c\Big(\int |\Gamma u|^2\Big)^{{p_1}/2}
+ c\Big(\int u^2\Big)^{{p_1}/2}.
\end{align*}
Taking ${p_1}^{th}$ roots, we have
$$\norm{u}_{L^{p_1}(\R^d)}\le c\norm{\Gamma u}_{L^2(\R^d)}+c\norm{u}_{L^2(\R^d)}.$$

If $2\le r\le {p_1}$, there exists $\theta\in [0,1]$ depending only
on $r$ and $p_1$ such that $\norm{u}_{L^r}
\le \norm{u}_{L^2}^\theta \norm{u}_{L^{p_1}}^{1-\theta}$; see, e.g.,
Proposition 6.10 of \cite{Folland}.
Combining with the inequality
$a^\th b^{1-\th}\le a+b$ yields
$$\norm{u}_{L^{r}}\le \norm{u}_{L^2}+ \norm{u}_{L^{p_1}}.$$
 We thus obtain
$$\norm{u}_{L^{r}(\R^d)}\le c\norm{\Gamma u}_{L^2(\R^d)}+c\norm{u}_{L^2(\R^d)}.$$
Applying this with $r $ first equal to $p$ and then with $r$ equal to
$2p/(4-p)$ and using Proposition \ref{THE2.2}, we obtain (1).

Suppose now that $u\in \sD({\sL})$ and that $h={\sL} u$. Let $\{E_\lam\}$ be
the spectral resolution of the operator $-{\sL}$. Then for $u\in L^2$,
$$u=\int_0^\infty \, dE_\lam u, \qq
\norm{u}_{L^2(\R^d)} =\int_0^\infty \, d(E_\lam u, E_\lam u).$$
If $u\in \sD({\sL})$ and $h={\sL}u$, then 
$$h=\int_0^\infty \lam \, dE_\lam u,\qq
\norm{h}_{L^2(\R^d)}=\int_0^\infty \lam^2\, d(E_\lam u, E_\lam u).$$
It then follows that
\begin{align*}
\norm{\Gamma u}_{L^2(\R^d)}^2&=\sE(u,u)\\
&=\int_0^\infty \lam\, d(E_\lam u,E_\lam u)\\
&=\int_0^1 \lam\, d(E_\lam u,E_\lam u)
+\int_1^\infty \lam\, d(E_\lam u,E_\lam u)\\
& \le \int_0^1 \, d(E_\lam u,E_\lam u)
+\int_1^\infty \lam^2\, d(E_\lam u,E_\lam u)\\
&\le \norm{u}^2_{L^2(\R^d)}+ \norm{h}^2_{L^2(\R^d)}.
\end{align*}
This and (1)  prove (2).
\end{proof}

\section{Strong stability}

Let
$$G(x)=\sup_{y\in \R^d}|\wt{A}(x,y)-A(x,y)|.$$

\begin{theorem} \label{THE3.2} Suppose $d>\al$. There exist 
$q\ge 2d/\al$ and 
a constant $c_1$ depending on $\Lambda,d,\alpha$, and $q$ such that
 if $f\in L^2(\R^d)$, then 
\bee\label{T5.1E1}
\|P_tf-\wt{P_t}f\|_{L^2}^2\le c_1\ (t^{-\frac12}+t^{\frac12})\|G\|_{L^{2q}}\norm{f}_{L^2}^2.
\eee
\end{theorem}

\begin{proof} For $t> 0$, let $u= P_tf- \wt{P}_tf$. By Lemma \ref{Grandlemma}(1), we know that $P_tf$ and $\wt{P}_tf$ are both in $\sF= W^{\frac{\alpha}{2},2}(\R^d)$, so
$u\in W^{\frac{\alpha}{2},2}(\R^d)$.

We write
\begin{align*}
\|P_tf-\wt{P}_tf\|_{L^2}^2&= (P_tf-\wt{P}_tf, u)\\
&= \int_0^t\frac{d}{ds}(P_s\wt{P}_{t-s}f, u)\, ds.\\
\end{align*}
This, Lemma \ref{Grandlemma}(3), and 
routine calculations show that
\begin{align}
\|P_tf-\wt{P}_tf\|_{L^2}^2= \int_0^t\Big(-\sE(\wt{P}_{t-s}f,P_su)+ \wt{\sE}(\wt{P}_{t-s}f,P_su)\Big)\, ds.
\label{BE23}
\end{align}

Using \eqref{BE23}, Lemma \ref{Grandlemma}(1) and H\"older's inequality, we obtain
{\allowdisplaybreaks
\begin{align}
\|P_tf&-\wt{P_t}f\|_{L^2}^2\nn\\
&= \int_0^t\Big(-\sE(\wt{P}_{t-s}f,P_su)+\wt{\sE}(\wt{P}_{t-s}f,P_su)\Big)\,ds\nn\\
&= \int_0^t\int_{\R^d}\int_{\R^d}\Big(\wt{P}_{t-s}f(y)-\wt{P}_{t-s}f(x)\Big)\Big(P_su(y)-P_su(x)\Big)\nn\\
&\qq \times\frac{\wt{A}(x,y)-A(x,y)}{|x-y|^{d+\alpha}}\, dy\, dx\, ds\nn\\
&\le  c\int_0^t\Big[\int_{\R^d}\int_{\R^d}\Big(\wt{P}_{t-s}f(y)-\wt{P}_{t-s}f(x)\Big)^2\frac{1}{|x-y|^{d+\alpha}}\, dy\, dx\Big]^{\frac12}\nn\\
& \qq\times\Big[\int_{\R^d}\int_{\R^d}\Big(P_su(y)-P_su(x)\Big)^2\frac{|\wt{A}(x,y)-A(x,y)|^2}{|x-y|^{d+\alpha}}\, dy\, dx\Big]^{\frac{1}{2}}\,ds\nn\\
&\le  c\int_0^t\Big[\int_{\R^d}\int_{\R^d}\Big(\wt{P}_{t-s}f(y)-\wt{P}_{t-s}f(x)\Big)^2\frac{\wt{A}(x,y)}{|x-y|^{d+\alpha}}\, dy\, dx\Big]^{\frac12}\nn\\
& \qq\times\Big[\int_{\R^d}\int_{\R^d}\Big(P_su(y)-P_su(x)\Big)^2\frac{|\wt{A}(x,y)-A(x,y)|^2}{|x-y|^{d+\alpha}}\, dy\, dx\Big]^{\frac{1}{2}}\,ds\nn\\
&\le c \int_0^t\Big[\wt{\sE}(\wt{P}_{t-s}f,\wt{P}_{t-s}f)\Big]^{\frac12}\nn\\
&\qq\times\Big[\int_{\R^d}\int_{\R^d}\frac{\Big(
P_su(y)-P_su(x)\Big)^2}{|x-y|^{d+\alpha}}\, dy\ G^2(x)\, dx\Big]^{\frac12}\,ds\nn\\
&\le c\int_0^t(t-s)^{-\frac{1}{2}}\|f\|_{L^2}\nn\\
& \qq \times\Big\{\int_{\R^d}\Big[\int_{\R^d}\frac{\Big(P_su(y)-P_su(x)\Big)^2}{|x-y|^{d+\alpha}}\, dy\Big]^{p'}\, dx\Big\}^{\frac{1}{2{p'}}}\\
&\qq\times \Big\{\int_{\R^d} G^{2{q'}}(x)\, dx\Big\}^{\frac{1}{2{q'}}}\,ds\nn\\
&= c\norm{f}_{L^2}\norm{G}_{L^{2{q'}}} \int_0^t(t-s)^{-\frac{1}{2}} \|\Gamma (P_su)(x)\|_{L^{2{p'}}}\, ds, \label{E3.1A}
\end{align}
}
where $p'$ and $q'$ are conjugate exponents.

We choose $p'$ so that $2p'$ is equal to the $p$ in Theorem \ref{meyers}(2).
By that
theorem,
\bee\label{E41}
\norm{\Gamma(P_su)}_{L^{2p'}}\le
c\norm{P_su}_{L^2}+c\norm{{\sL}(P_su)}_{L^2}.
\eee

Since
$P_s,P_t,$ and $\wt P_t$ are contractions, 
\bee\label{E42}
\norm{P_su}_{L^2}\le \norm{u}_{L^2}=\norm{P_tf-\wt P_t f}_{L^2}
\le 2\norm{f}_{L^2}.
\eee
To estimate ${\sL}(P_su)$, we note $P_{s/2}u\in \sD({\sL})$ by Lemma \ref{Grandlemma}(2)
and then use Lemma  \ref{Grandlemma}(4).
Then 
\begin{align}
\norm{{\sL}(P_su)}_{L^2}&=\norm{(-{\sL})^{1/2} P_{s/2}(-{\sL})^{1/2} (P_{s/2}u)}_{L^2}
\label{E43}\\
&\le cs^{-1/2} \norm{(-{\sL})^{1/2}(P_{s/2}u)}_{L^2}\nn\\
&=cs^{-1/2}\sE(P_{s/2}u,P_{s/2}u)^{1/2}\nn\\
&\le cs^{-1/2} \sE(u,u)^{1/2}\nn\\
&\le cs^{-1/2}\big[\sE(P_tf,P_tf)^{1/2}+\sE(\wt P_t f , \wt P_tf)^{1/2}\big]\nn\\
&\le c(st)^{-1/2}\norm{f}_{L^2},\nn
\end{align}
where Lemma \ref{Grandlemma}(1) is used in the first and last inequalities.
Combining \eqref{E3.1A}, \eqref{E41}, \eqref{E42}, and \eqref{E43} yields
our result.
\end{proof}

\begin{remark}\label{R-scaling}{\rm
A scaling argument allows one to improve \eqref{T5.1E1} to  
\bee\label{T5.1E2}
\|P_tf-\wt{P_t}f\|_{L^2}^2\le c_1\ t^{-d/2q\al}\|G\|_{L^{2q}}\norm{f}_{L^2}^2.
\eee 
We give a sketch and leave the details to the reader.

If $X_t$ is the strong Markov process whose semigroup is $P_t$, let
$Y_t=aX_{a^{-\al} t}$. Routine calculations shows that the semigroup $Q_t$
for $Y$ is related to that of $X$ by the equation
$$P_tf(x)=Q_{a^\al t}g(ax),$$
where $g(z)=f(z/a)$, and that the Dirichlet form of $Y$ is given by
$$\sE_Y(f,f)=\int_{\R^d}\int_{\R^d} \frac{(f(y)-f(x))^2}{|x-y|^{d+\al}}
B(x,y)\, dy\, dx,$$
where $B(x,y)=A(x/a,y/a)$ and $A$ is the  function in \eqref{dfno}.

Suppose we define $\wt Q_t$ and $\wt B$ in terms of $\wt P_t$ similarly
and let 
$$H(x)=\sup_{y\in \R^d} |B(x,y)-\wt B(x,y)|.$$ Fix $t$ and set $a=t^{-1/a}$
so that $a^\al=t^{-1}$. A straightforward calculation and an application
of Theorem \ref{THE3.2} yield  
$$\norm{P_tf-\wt P_tf}_{L^2}^2=a^{-d} \norm{Q_1g-\wt Q_1g}_{L^2}^2\le ca^{-d}\norm{H}_{L^{2q}}
\norm{g}_{L^2}^2,$$
where $g(z)=f(z/a)$. Further calculations show that
$$\norm{H}_{L^{2q}}=a^{d/2q}\norm{G}_{L^{2q}}$$
and 
$$\norm{g}_{L^2}^2=a^d\norm{f}_{L^2}^2.$$
Combining gives \eqref{T5.1E2}.
}
\end{remark}

Let $p(t,x,y)$ and $\wt{p}(t,x,y)$ be the heat kernels corresponding to $P_t$ and $\wt{P_t}$.
By Theorem 4.14 in \cite{Chen-Kumagai}, we know there exist $\gamma>0$ and
a constant $c_1$ such that
\bee\label{CE1P}
|p(t,x,y)-p(t,z,v)|\le c_1\ t^{-\frac{d+\gamma}{\alpha}}(|x-z|+|y-v|)^{\gamma}
\eee
for all $x,y,z,v\in \R^d$.
By Theorem 1.1 in \cite{Chen-Kumagai}, there exist constants $c_2$ and $c_3$ such that
\begin{align}
c_2\min\Big{\{t^{-\frac{d}{\alpha}},\frac{t}{|x-y|^{d+\alpha}}\Big\}}&\le p(t,x,y)\nn\\
&\le c_3\min{\Big\{t^{-\frac{d}{\alpha}},\frac{t}{|x-y|^{d+\alpha}}\Big\}}
\label{CE1}
\end{align}
for all $x,y\in\R^d$.

We have the following two theorems. Once we have Theorem \ref{THE3.2},
\eqref{CE1P}, and \eqref{CE1},
the proofs are so similar to the corresponding theorems in 
\cite{Chenetal} that we refer the reader to that paper for the proofs.

\begin{theorem}\label{THE4.2A} Let $t> 0$. There exist $q>1$ and  a constant $c_1$
depending on $t, \Lambda, \gamma,  d, \alpha,$ and $q$ such that for any $x, y\in \R^d$
\[|p(t,x,y)- \wt{p}(t,x,y)|\le c_1 \|G\|_{2q}^{\frac{\gamma}{2(d+\gamma)}}.\]
\end{theorem}

\ms

\begin{theorem}\label{THE4.3A} Let  $t>0$. There exist $q>1$ and 
 a constant $c_2$ depending on $t, \Lambda, \gamma,  d,\alpha$, and $q$ such that for any $p\in [1,\infty]$, we have
\[\|P_tf-\wt{P}_tf\|_{L^p}\le c_2 \|G\|_{2q}^{\frac{\gamma \alpha}{2(d+\gamma)(d+\alpha)}}\norm{f}_{L^p}.\]
\end{theorem}

As in Remark \ref{R-scaling}, one could use scaling        
to obtain an explicit bound on how the constants
depend on $t$. We leave this to the interested reader.

\section{Proof of Lemma \ref{Grandlemma}}\label{proofoflemma}

In this section we give a proof of the lemma stated in  Section  \ref{prelim}.

Let $\{E_\lam\}$, $\lam\ge 0$, be the spectral representation of $-{\sL}$.
For $f\in \sF$, we have 
$$\sE(f,f)=\int_0^\infty \lam\, d(E_\lam f, E_\lam f);$$
see \cite{FOT}.

\begin{proof}[Proof of Lemma \ref{Grandlemma}]
(1) This follows from
\begin{align*}
\sE(P_tf,P_tf)&=\int_0^\infty \lam e^{-2\lam t}
\, d(E_\lam f, E_\lam f)\\
&\le ct^{-1}\int_0^\infty \, d(E_\lam f, E_\lam f)=ct^{-1}
\norm{f}_2^2,
\end{align*}
since $\lam e^{-2\lam t}\le ct^{-1}$ for all $\lam\ge 0$.

(2) By the spectral representation of $-{\sL}$, 
we have
$$\frac{P_h(P_tg)-P_tg}{h}=\frac{P_{t+h}g-P_tg}{h}=\int_0^\infty
\frac{e^{-\lam(t+h)}-e^{-\lam t}}{h}\, dE_\lam g.$$
Let $H=-\int_0^\infty \lam e^{-\lam t}\, dE_\lam g$. Note 
$\norm{H}_{L^2}$ is finite because $\lam^2 e^{-2\lam t}$ is bounded.
Then
$$\Big\| \frac{P_h(P_tg)-P_tg}{h}-H\Big\|_{L^2}^2
=\int_0^\infty \Big[ \frac{ e^{-\lam(t+h)}-e^{-\lam t}}{h}+\lam e^{-\lam t}
\Big]^2\, d(E_\lam g, E_\lam g),$$
which tends to 0 as $h\to 0$ by dominated convergence. Therefore
$P_tg\in\sD({\sL})$ and ${\sL}(P_tg)=H$.

(3)  For any $g\in \sF$, we have
\[(P_tf, g)= \int_0^{\infty}e^{-\lambda t}\, d(E_{\lambda}f, g),\]
and so
\[\frac{d}{dt}(P_tf, g)= -\int_0^{\infty}\lambda e^{-\lambda t}\, d(E_{\lambda}f, g).\]
On the other hand,
\[\sE(P_tf, g)= \int_0^{\infty}\lambda \,d(E_\lambda P_tf, g)= \int_0^{\infty}\lambda e^{-\lambda t}\, d(E_{\lambda}f, g),\]
which proves the assertion.

(4) We prove this by writing
$$\int_0^\infty \lam e^{-2\lam t}\, d(E_\lam f, E_\lam f)
\le \int_0^\infty \lam \, d(E_\lam f, E_\lam f),$$
which translates to \eqref{Pr44}.
\end{proof}

\ni {\bf Richard F. Bass}\\
Department of Mathematics\\
University of Connecticut \\
Storrs, CT 06269-3009, USA\\
{\tt r.bass@uconn.edu}
\ms

\ni {\bf Hua Ren}\\
Department of Mathematics\\
University of Connecticut \\
Storrs, CT 06269-3009, USA\\
{\tt hua.ren@uconn.edu}
\ms

\end{document}